\documentclass {amsart}
\usepackage{amssymb}

\newtheorem{thm}{Theorem}[section]
\newtheorem{lem}[thm]{Lemma}
\newtheorem{cor}[thm]{Corollary}

\numberwithin{equation}{section}

\newcommand{\ra}{{\rightarrow}}

\newcommand{\gen}{\text{gen}}

\newcommand{\z}{{\mathbb Z}}
\newcommand{\q}{{\mathbb Q}}

\newcommand{\s}{{\mathfrak s}}
\newcommand{\n}{{\mathfrak n}}

\begin{document}
\title{Ternary universal sums of generalized pentagonal numbers}


\author{Byeong-Kweon Oh}
\address{Department of Mathematical Sciences, Seoul National University,
 Seoul 151-747, Korea}
\email{bkoh@math.snu.ac.kr}





\begin{abstract}
For any $m\ge3$, every integer of the form $p_m(x)=\frac{(m-2)x^2-(m-4)x}2$ with $x \in \z$
is said to be a generalized $m$-gonal number. Let $a\le b\le c$ 
be positive integers. For every non negative integer $n$, if there are 
integers $x,y,z$ such that $n=ap_k(x)+bp_k(y)+cp_k(z)$, then  the quadruple $(k;a,b,c)$
is said to be {\it universal}. Sun gave in \cite{s1} all possible quadruple candidates that are universal and
proved some quadruples to be universal (see also \cite{gs}). He remains the following 
quadruples $(5,1,1,k)$ for $k=6,8,9,10$, $(5,1,2,8)$, and $(5,1,3,s)$ for $7 \le s \le 8$ as candidates and
conjectured the universality of them.
In this article we prove that the remaining $7$ quadruples given above are, in fact, universal. 
\end{abstract}

\maketitle

\section{Introduction}
Let $m$ be any positive integer greater than two. 
An integer of the form $p_m(x)=\frac{(m-2)x^2-(m-4)x}2$ for any non negative integer $x$ is said to be a {\it polygonal number of order $m$} (or $m$-gonal number).
If the variable $x$ is an integer, $p_m(x)$ is called a {\it generalized polygonal number of order $m$} (or generalized $m$-gonal number). 
From the definition every $m$-gonal number is a generalized $m$-gonal number and the converse is also true for $m=3$ or $4$. However
if $m$ is greater than $4$, the set of all $m$-gonal numbers is a proper subset of the set of all generalized $m$-gonal numbers. 

In 1638 Fermat asserted that every non negative integer can be written as a sum of $m$ polygonal numbers of order $m$. This 
was proved by Lagrange, Gauss and Cauchy in the cases $m=4$, $m=3$ and $m \ge 5$, respectively (see Chapter 1 of \cite{n}). 

The polygonal number theorem stated above was generalized in many directions. 
For example Lagrange's four square theorem was generalized to find all quaternary quadratic forms that represent all non negative integers. 
Recently Bhargava and Hanke in \cite{bh} finally completed this problem by proving, so called, 290-theorem, which is a generalization of the Conway and Schneeberger's 15-theorem (see \cite{b}). 
Gauss' triangular theorem was first generalized by Liouville. 
To state the theorem more precisely we begin with defining the terminology {\it universality}. 
Let $a\le b\le c$ be positive integers. A ternary sum of polygonal numbers $ap_i(x)+bp_j(y)+cp_k(z)$ is said to be
{\it universal over $\mathbb N$} if the equation $n=ap_i(x)+bp_j(y)+cp_k(z)$ has a non negative integer solution $x,y,z$ for any non negative integer $n$. 
In particular
if it has an integer solution $x,y,z$, it is said to be {\it universal over $\z$}.    
In 1862 Liouville determined all ternary universal sums of polygonal numbers in the case when $i=j=k=3$. When $i,j,k \in \{3,4\}$, 
this problem was done by Sun and his collaborators (see \cite{gps}, \cite{os} and \cite{s1}). 
Recently Sun gave in \cite{s2} a complete list of candidates of all possible ternary universal sums of (generalized) polygonal numbers.  
In particular he proved that there are at most 20 ternary universal sums over $\z$ which is of the form $P^k_{a,b,c}(x,y,z)=ap_k(x)+bp_k(y)+cp_k(z)$ 
for $k=5$ or $k \ge 7$, and conjectured that these are all universal over $\z$. 
They are, in fact, $k=5$ and 
$$
\begin{array} {lll}
(a,b,c)=& (1,1,s) \quad 1\le s \le 10, \ \ s\ne 7,\\
&(1,2,2), \  \ (1,2,3), \  \ (1,2,4), \ \ (1,2,6), \ \ (1,2,8),\\
&(1,3,3), \  \ (1,3,4), \ \ (1,3,6), \  \ (1,3,7), \  \ (1,3,8), \ \ (1,3,9).\\
\end{array}
$$ 
Guy realized in \cite{g} that $P^5_{1,1,1}$ is universal over $\z$ (for the complete proof, see \cite{s2}).   
In the same article \cite{s2} as above, Sun also proved the universality on the cases when $k=5$ and
$$
(a,b,c)=(1,1,2), \ \ (1,1,4), \ \ (1,2,2), \ \ (1,2,4), \ \ (1,1,5), \ \ (1,3,6).
$$ 
Note that the set of generalized pentagonal numbers is equal to the set of triangular numbers.  Shortly after publishing this result
Ge and Sun   proved in \cite{gs} the universality on the cases  when $k=5$ and
$$
(a,b,c)=(1,1,3), \ \ (1,2,3), \ \ (1,2,6), \ \ (1,3,3), \ \ (1,3,4), \ \ (1,3,9).
$$
Therefore it remains seven candidates of ternary sums of pentagonal numbers which are universal over $\z$. 

In this article we prove that the remaining seven candidates are all universal over $\z$. One may easily show that
$P^5_{1,b,c}(x,y,z)=n$ has a solution over $\z$ if and only if
\begin{equation}
(6x-1)^2+b(6y-1)^2+c(6z-1)^2=24n+b+c+1
\end{equation}
has a solution over $\z$. If an integer $w$ is relatively prime to $6$, then $w$ or $-w$ is congruent to $-1$ modulo $6$. 
Therefore $P^5_{1,b,c}$ is universal over $\z$ if and only if the ternary quadratic form 
\begin{equation}
x^2+by^2+cz^2=24n+b+c+1 
\end{equation} 
has an integer solution $x,y,z$ with $\gcd(xyz,6)=1$. In many cases the problem can be reduced to the problem of representations  
of ternary quadratic forms without congruence condition on the solution. For example assume that $\gcd(b+c+1,6)=1$. 
Then the equation (1.1) has an integer solution if and only if the equation
\begin{equation}
x^2+b(x-6y)^2+c(x-6z)^2=24n+b+c+1
\end{equation}
has an integer solution. Under this situation we use the method on the representations of ternary forms which is developed in \cite{o1}.
It seems to be quite a difficult problem to determine whether or not $p^k_{a,b,c}$ is universal over $\mathbb N$ for $k\ge 5$.

The term lattice will always refer to an integral $\z$-lattice on an
$n$-dimensional positive definite quadratic space over $\q$. The
scale and the norm ideal of a lattice $L$ are denoted by $\s(L)$ and
$\n(L)$ respectively. Let $L=\z x_1+\z x_2+\cdots +\z x_n$ be a
$\z$-lattice of rank $n$.
 We write
 $$
 L \simeq (B(x_i,x_j)).
$$
The right hand side matrix is called a {\it matrix presentation} of
$L$. If $(B(x_i,x_j))$ is diagonal, then we simply write $L \simeq \langle Q(x_1),\dots,Q(x_n)\rangle$.

Throughout this paper, we always assume that every $\z$-lattice $L$
is {\it positive definite} and is {\it primitive} in the sense that
$\s(L)=\z$. 

For any $\z$-lattice $L$, $Q(\gen(L))$ ($Q(L)$) is denoted by the
set of all integers that are represented by the genus of $L$ ($L$
itself, respectively). In particular, we call an integer $a$ {\it
eligible} if $a \in Q(\gen(L))$ following Kaplansky. For any integer $a$, $R(a,L)$
is denoted by the set of all vectors $x \in L$ such that $Q(x)=a$ and $r(a,L)=\vert R(a,L)\vert$.

Any unexplained notations and terminologies can be found in \cite{ki} or \cite{om2}.

\section{General Tools}

For a positive integer $d$ and a non-negative integer $a$, we define
$$
S_{d,a}=\{ dn+a \mid n \in \z^{+}\cup \{0\} \}.
$$
A $\z$-lattice $L$ is called {\it $S_{d,a}$-universal} if it
represents every integer in the set $S_{d,a}$. $L$ is called  {\it
$S_{d,a}$-regular} if it represents every integer in $S_{d,a}$ that
is represented by the genus of $L$, and at least one integer in
$S_{d,a}$ is represented by the genus of $L$. 
Hence we have
$$
\text{$L$ is $S_{d,a}$-universal  if and only if $S_{d,a} \subset
Q(L)$}
$$
and
$$
\text{$L$ is $S_{d,a}$-regular if and only if $\varnothing \ne
Q(\gen(L))\cap S_{d,a} \subset Q(L)$}.
$$
For the finiteness theorem of ternary $S_{d,a}$-universal and $S_{d,a}$-regular lattices, see \cite{o2}.

Let $M$ and $N$ be ternary $\z$-lattices on the quadratic space $V$. For any positive integer $d$ and an integer $a$ such that $0\le a<d$, we define
$$
R(N,d,a)=\{ x \in N/dN : Q(x) \equiv a \pmod {d}\}
$$
and 
$$
R(M,N,d)=\{ \sigma \in O(V) : \sigma(dN) \subset M\}.
$$
A vector $x \in R(N,d,a)$ is said to be  {\it good} (with respect to $M,N,d$ and $a$) if 
there is a $\sigma \in R(M,N,d)$ such that $\sigma(\widetilde{x}) \in M$ for any $\widetilde{x} \in N$ satisfying $\widetilde{x}  \equiv x \pmod d$. The subset of all 
good vectors in $R(N,d,a)$ is denoted by $R_M(N,d,a)$. If $R(N,d,a)=R_M(N,d,a)$,  
then we write
$$
N \prec_{d,a} M.
$$
Note that if  $N \prec_{d,a} M$, then as stated in \cite{o1},
$$
S_{d,a} \cap Q(N) \subset Q(M).  
$$
 Note that the converse is not true in general. One may also easily show that
\begin{itemize}

\item [(i)] if $N$ is $S_{d,a}$-universal, so is $M$;

\item [(ii)] if the assumption holds for any $N \in \gen(M)$ and $S_{d,a} \cap Q(\gen(M)) \ne \varnothing$, then $M$ is $S_{d,a}$-regular;

\item [(iii)] if $M \ra \ \gen(N)$, $S_{d,a} \cap Q(\gen(M)) \ne \varnothing$, and $N$ is $S_{d,a}$-regular, then so is $M$.    
 
\end{itemize}

\begin{lem} \label{lem1} Let $V$ be a ternary quadratic space. For an isometry $\sigma \in O(V)$, define
$$
V_{\sigma}=\{ x \in V : \text{there is a positive integer $k$ such that } \sigma^k(x)=x\}.
$$   
If $\sigma$ has an infinite order, then $V_{\sigma}$ is one dimensional subspace of $V$ and $\sigma(x)=\det(\sigma)x$ for any $x \in V_{\sigma}$. 
\end{lem}

\begin{proof} If $\sigma^k(x)=x$ and $\sigma^s(y)=y$ for some $s$ and $y \not \in \q x$. Then the fixed space
of $\sigma^{ks}$ is two dimensional and hence it is a reflection. This contradicts the fact that $\sigma$ 
has an infinite order. Note that every isometry $\sigma$ of a ternary quadratic space has an eigenvalue $\det(\sigma)$. 
\end{proof}

\begin{cor} \label{cor1} Under the same notations as above, assume that there is a partition $R(N,d,a)-R_M(N,d,a)=P_1\cup\cdots\cup P_k$
satisfying the following properties: for each $i=1,\dots,k$, there is a $\tau_i \in O(V)$ such that

\noindent {\rm (i)} $\tau_i$ has an infinite order;
 
\noindent {\rm (ii)} $\tau_i(dN) \subset N$;

\noindent {\rm (iii)} $\tau_i(x) \in N$ and  $\tau_i(x) (\text{mod } d) \in P_i \cup R_M(N,d,a)$ for any $x \in P_i$.

Then we have
$$     
(S_{d,a} \cap Q(N))-\cup_{i=1}^k Q(z_i)\z^2 \subset Q(M),
$$
where the vector $z_i$ is a primitive vector in $N$ which is an eigen vector of $\tau_i$.   
\end{cor}

\begin{proof} Without loss of generality we may assume that $k=1$. 
Assume that $Q(x) \equiv a \pmod d$ for $x \in N$. If $x (\text{mod } d) \in R_M(N,d,a)$,
then $Q(x)$ is represented by $M$. So we may assume that $x (\text{mod } d) \in R(N,d,a)-R_M(N,d,a)$. 
Then $\tau_1(x) \in N$ from the assumption. Hence if $\tau_1(x) (\text{mod } d) \in R_M(N,d,a)$, then we are done,
otherwise $\tau_1^2(x) \in N$.
Inductively we may assume that $\tau_1^m(x) \not \in R_M(N,d,a)$ for any positive integer $m$. Therefore 
if $x$ is not an eigen vector of $\tau_1$, the set $\{ \tau_1^m(x) \in N : m\ge 0\}$ is an infinite set by Lemma \ref{lem1}. This is a contradiction
to the fact that the number of representations of any integer by a ternary form is finite.    
\end{proof}

Finding all ternary $\z$-lattices satisfying $Q(\gen(L))=Q(L)$, which are called {\it regular lattices}, has long and rich history. 
In 1997 Kaplansky and his collaborators \cite{jks}  provided a list of $913$ candidates of primitive positive definite
regular ternary  forms up to equivalence and stated that there are no others. All but
$22$ of $913$ are already verified to be regular. Recently the author proved in \cite{o1} the regularity of eight ternary lattices among
$22$ candidates. There are also some examples of ternary lattices satisfying $\vert Q(\gen(L))-Q(L)\vert=1$ (see, for example, \cite{j}).
 However there  are no known examples of ternary lattices such that $\vert Q(\gen(L))-Q(L)\vert\ge2$ and $Q(L)$ is completely determined
 except Ono and Soundararajan's result  \cite{os1}. In that article they completely determined the set of all integers that are 
represented by the Ramanujan's ternary form $\langle 1,1,10\rangle$ under the generalized Riemann hypothesis for some Dirichlet $L$-functions.    
The following theorem provides such an example without any assumption. 

 Every computation such as $R(N,d,a)$ and $R_M(N,d,a)$ for some $M, N, d$ and $a$ was done by the computer program MAPLE.   
 
\begin{thm} \label{thm1} For the ternary quadratic form $M=\langle 1\rangle \perp \begin{pmatrix} 9&3\\3&10\end{pmatrix}$,
$$
Q(M)=Q(\gen(M))-\{ 2\cdot2^{2m}, 5\cdot2^{2n} : m,n\ge 0\}.
$$  
\end{thm}

\begin{proof}  Note that $h(M)=3$. Let $a \in Q(\gen(M))$. First assume that $a$ is divisible by $3$. 
Since $M_3 \simeq \langle 1,1,3^4\rangle$, $a=9b$ for a positive integer $b$. Furthermore since $b$ is represented by 
$\langle 1,1,9\rangle \simeq \langle 1 \rangle \perp \begin{pmatrix} 1&1\\1&10\end{pmatrix}$, $a=9b$ is represented by
$\langle 9 \rangle \perp \begin{pmatrix} 9&9\\9&90\end{pmatrix}$, which is a subform of $M$. Therefore
$a$ is represented by $M$. 

Assume that  $ a \equiv 1 \pmod 3$. If we define $N=\langle 1,1,9\rangle$, then
$$
R(N,3,1)=\{(0,\pm1,s),(\pm1,0,t) : s,t \in \mathbb F_3\},
$$
and
$$
\left\lbrace\begin{pmatrix}0&1&0\\\frac13&0&\frac{-1}3\\0&0&1\end{pmatrix}, \begin{pmatrix}0&1&0\\\frac13&0&\frac13\\0&0&-1\end{pmatrix},
\begin{pmatrix}1&0&0\\0&\frac{-1}3&\frac13\\0&0&-1\end{pmatrix},\begin{pmatrix}1&0&0\\0&\frac13&\frac13\\0&0&-1\end{pmatrix}\right\rbrace \subset R(M,N,3).
$$
In fact $\vert R(M,N,3)\vert=16$, however we need only four given above. One may easily show that $R(N,3,1)=R_M(N,3,1)$ and 
hence $N \prec_{3,1} M$. Since $h(N)=1$, $a$ is represented by $N$ and 
is also represented by $M$.

Finally assume that $a \equiv 2 \pmod 3$. Let $N=\begin{pmatrix} 2&1\\1&5\end{pmatrix} \perp \langle9\rangle.$
Note that $N$ represents $a$. In this case we have 
$$
R(N,3,2)=\{ (0,\pm1,s), (\pm1,0,t),(\pm1,\pm2,u) : s,t,u \in \mathbb F_3\}
$$
and
$$
\left\lbrace\begin{pmatrix}-1&1&0\\\frac13&\frac23&\frac13\\0&0&-1\end{pmatrix},\begin{pmatrix}1&-1&0\\\frac{-1}3&\frac{-2}3&\frac13\\0&0&-1\end{pmatrix}\right\rbrace \subset R(M,N,3).
$$
Therefore we have
$$
R(N,3,2)-R_M(N,3,2)=\{(\pm1,0,0),(0,\pm1,0),(\pm1,\pm2,0)\}.
$$ 
Let $P_1=\{ (\pm1,0,0)\}, P_2=\{(0\pm1,0)\}$ and
$P_3=\{(\pm1,\pm2,0)\}$, and 
$$
\tau_1=\displaystyle \frac 13\begin{pmatrix} -3&-2&-2\\0&1&4\\0&-2&1\end{pmatrix}, \quad
\tau_2=\displaystyle \frac 13\begin{pmatrix} -2&0&5\\1&3&-1\\-1&0&-2\end{pmatrix}, \quad
 \tau_3=\displaystyle \frac 13\begin{pmatrix} -2&-2&-5\\1&-2&1\\-1&-1&2\end{pmatrix}.
$$		
Note that $\tau_1$ and $\tau_2$ satisfy all conditions in Corollary \ref{cor1}, however 
$$
\tau_3(x) \in R_M(N,3,2) \cup P_2 \quad \text{for any $x \in P_3$}.
$$
Hence we may still apply Corollary \ref{cor1}. Note that the primitive eigen vectors for each $\tau_i$ for $i=1,2$ are $(1,0,0)$ and $(0,1,0)$, respectively.  
Therefore if $a$ is not of the form $Q_M(\pm t,0,0)=2t^2$ or $Q_M(\pm t,\mp t,0)=Q(0,\pm t,0)=5t^2$, then $a$ is also represented by $M$.
Furthermore even if $a$ is one of those forms, $a$ is represented by $M$ except only when $r(a,N) \le4$.   
Assume that $a=2t^2$ or $5t^2$ for a positive integer $t$ and there is an odd prime $p \ge 5$ dividing $t$. Since 
$$
r(2p^2,N)=2p+2-2\left(\frac{-2}p\right), \quad r(5p^2,N)=4p+4-4\left(\frac{-2}p\right), \quad r(125,N)=24 
$$  
by Minkowski-Siegel formula, there is a representation which is not contained in the eigen space for both cases. Therefore 
$a$ is represented by $M$ if $t$ has an odd prime divisor.  Finally $M_2$ is anisotropic and $2,5$ is not represented by $M$. 
Hence every integer of the form $2\cdot 2^{2m}$ or $5\cdot 2^{2n}$ is not represented by $M$.    		
\end{proof}

\begin{lem} \label{lem2}  Assume that a binary form $L=\langle 1,k\rangle$ represents a positive integer $pN$ for an odd prime $p$ not dividing $k$ 
and a positive integer $N$. If $L$ represents $p$ or $r(p^2,L)>2$ then there are positive integers $x,y$ not divisible by $p$ 
such that $x^2+ky^2=pN$.  
\end{lem}

\begin{proof} When $L$ represents $p$, the proof was given by Jones in his unpublished Ph. D. dissertation (see also \cite{j}).
Though the proof of the remaining case is quite similar to this, we present the proof for the completeness.
First define 
$$
\Phi_p(L)=\{ a : \text{$x^2+ky^2=a$ has an integer solution $x,y$ such that $\gcd(xy,p)=1$}\}.
$$  
Note that if $a \in \Phi_p(L)$, then $at^2 \in \Phi_p(L)$ for any positive integer $t$ not divisible by $p$ and  $ka \in \Phi_p(L)$.
Assume that $S,T \in \Phi_p(L)$ and $ST \equiv 0 \pmod p$. 
Let
$$
A^2+kB^2=S, \qquad C^2+kD^2=T  \qquad \text{and} \quad ABCD \not \equiv 0 \pmod p. 
$$
Then $(AC\pm kBD)^2+k(AD\mp BC)^2=ST$ and at least one of $AC+kBD$ or $AC-kBD$ is not divisible by $p$. Hence $ST \in \Phi_p(L)$.  

Now from the assumption we know that $p^2 \in \Phi_p(L)$, that is, there are integers  $x,y$  such that $x^2+ky^2=p^2$ with $\gcd(xy,p)=1$. 
Assume that $A^2+kB^2=pN$ and $A \equiv B \equiv 0 \pmod p$. If one of $A$ or $B$ is zero, then clearly $pN \in \Phi_p(L)$. 
Assume that $A=p^ma$ and $B=p^nb$ and $ab \not \equiv 0 \pmod p$.  
Without loss of generality we may assume that $m \ge n$. Since  $(p^{m-n}a)^2+kb^2=p^{1-2n}N$,
$$
(p^{m-n}ax\pm kby)^2+k(p^{m-n}ay\mp bx)^2=p^{3-2n}N.
$$
 Now by choosing the sign suitably, we may assume that both terms in the left hand side are not divisible by $p$. Hence $p^{3-2n}N \in \Phi_p(L)$. 
Therefore $pN=p^{2(n-1)}p^{3-2n}N \in \Phi_p(L)$. This completes the proof.  
\end{proof}

\section{Ternary universal sums of pentagonal numbers}
In this section we prove that the remaining seven candidates of ternary sums of pentagonal numbers are universal over $\z$.  

\bigskip

\noindent {\bf (i) $p_5(x)+p_5(y)+6p_5(z)$.} First we show that for every eligible number $k$ of $F(x,y,z)=2x^2+4y^2+8z^2+2xy+2yz$,
$F(x,y,z)=k$ has an integral solution $(a,b,c) \in \z^3$ with $a \not \equiv c \pmod 2$ unless $k \equiv 0 \pmod 4$. Note that 
the class number of $F$ is $1$ and hence $F(x,y,z)=k$ has an integer solution. Suppose that every solution $(a,b,c)$ satisfies
$a \equiv c \pmod 2$. Then one may easily check that if $F(a,b,c)=k$, 
$$
F\left(\frac {a-2b-c}2,\frac {a+2b+c}2,-c\right)=k.
$$
Hence if we let $S=\frac 12 \begin{pmatrix}1&-2&-1\\1&2&1\\0&0&-2\end{pmatrix}$, which has an infinite order,
$S^m(a,b,c)^t$ is also an integer solution for any positive integer $m$. Therefore if $(a,b,c) \ne (t,-2t,7t)$
for any $t \in \z$, this is a contradiction by Corollary \ref{cor1}. Finally note that 
$$
F(t,-2t,7t)=F(2t,5t,5t)
$$
for any $t \in \z$. Hence if $k \not \equiv 0 \pmod 4$, then such a solution exists always.

Now let $(a,b,c)$ be an integer solution with $a \not \equiv c \pmod 2$ of the equation
$F(x,y,z)=6n+2$ for a non negative integer $n$. Note that every positive integer of the form $6n+2$ is an eligible integer of $F$.
Let 
$$
d=a+5c, \quad e=-b+c, \quad f=c.
$$
Then $6n+2=F(a,b,c)=2d^2+4e^2+54f^2-2de-18df$. Therefore
$$
\begin{array} {ll}
24n+8=4F(a,b,c)&=8d^2+16e^2+216f^2-8de-72df\\
&=d^2+(d-4e)^2+6(d-6f)^2.\\
\end{array}
$$
Since $d$ is odd,  the integers $d, d-4e$ and $d-6f$ are relatively prime to $6$. Therefore 
$24n+8=x^2+y^2+6z^2$ has an integer solution $(x,y,z)$ with $\gcd(xyz,6)=1$. 

\bigskip

\noindent{\bf (ii) $p_5(x)+p_5(y)+8p_5(z)$.} Let 
$$
M=\begin{pmatrix} 4&2&2\\2&5&1\\2&1&10\end{pmatrix}, \quad N=\begin{pmatrix} 5&1&2\\1&5&-2\\2&-2&8\end{pmatrix}.
$$
Then we have $R(N,6,5)-R_M(N,6,5)=\{ (\pm2,3,0),(3,\pm2,0)\}$. Let $P_1=\{ (\pm2,3,0)\}$, $P_2=\{(3,\pm2,0\}$, and
$$
\tau_1=\displaystyle \frac 16\begin{pmatrix} -6&-4&-2\\0&6&-6\\0&4&2\end{pmatrix}, \quad
 \tau_2=\displaystyle \frac 16\begin{pmatrix} -6&0&-6\\4&6&-2\\4&0&-2\end{pmatrix}.
$$
Then one may easily show that this information satisfies all conditions in the Corollary \ref{cor1} and the eigen vectors for each $\tau_i$ are
$z_1=(1,0,0)$ and $z_2=(0,1,0)$. Clearly $Q(z_1)=Q(z_2)=5$ is represented by $M$. Consequently $S_{6,5} \cap Q(N) \subset Q(M)$. 
Note that $h(N)=1, h(M)=4$ and every positive integer of the form $12n+5$ is an eligible integer of $N$. Therefore $M$ represents
every positive integer of the form $12n+5$. Now consider the following equation
$$
\begin{array} {ll}
24n+10&=x^2+(x-6y)^2+8(x-2z)^2 \\
&=10x^2+36y^2+32z^2-12xy-32xz=f(x,y,z).  
\end{array}
$$ 
Since $f$ is isometric to $2M$, the above equation has an integer solution $(x,y,z)$. Since
all of $x, x-6y, x-2z$ are relatively prime to $6$, the equation $24n+10=x^2+y^2+8z^2$ has an integer solution
$(x,y,z)$ with $\gcd(xyz,6)=1$. 

\bigskip

\noindent {\bf (iii) $p_5(x)+p_5(y)+9p_5(z)$.} 	Note that every positive integer of the form $24n+11$ is represented by $\langle 1\rangle \perp \begin{pmatrix} 9&3\\3&10\end{pmatrix}$
by Theorem \ref{thm1}. Therefore $24n+11=x^2+(z-3y)^2+9z^2$ has always an integer solution. Note that $\gcd(x(z-3y)z,6)=1$. 

\bigskip

\noindent {\bf (iv) $p_5(x)+p_5(y)+10p_5(z)$.} Let $M=\langle 1,4,5\rangle$ and $N=\langle 1,1,20\rangle$. Then one may easily show by computation
that $N \prec_{6,0} M$. Since $\gen(M)=\{M,N \}$, $M$ represents every positive integer of the form $12n+6$. Let $a,b,c$ be integers such that
$12n+6=a^2+4b^2+5c^2$. Clearly $a$ and $c$ are odd. Furthermore since $9=2^2+5\cdot 1^2$, 
we may further assume that $c$ is not divisible by $3$ by Lemma \ref{lem2}. Now if $x=a+2b, y=a-2b$ and $z=c$, then
$$
24n+12=2a^2+8b^2+10c^2=(a+2b)^2+(a-2b)^2+10c^2=x^2+y^2+10z^2.
$$  
Since exactly one of $a$ and $b$ is divisible by $3$, $\gcd(xyz,6)=1$.

\bigskip

\noindent {\bf (v) $p_5(x)+2p_5(y)+8p_5(z)$.}   Let 
$$
M=\begin{pmatrix} 1&0&0\\0&8&4\\0&4&10\end{pmatrix}, \quad 		
N=\begin{pmatrix} 3&1&1\\1&3&-1\\1&-1&9\end{pmatrix}, \quad
L=\begin{pmatrix} 1&0&0\\0&10&2\\0&2&58\end{pmatrix}.
$$
Note that $\gen(M)=\{M,N\}$ and $S_{24,11} \subset Q(\gen(M))$. Since $N \prec_{24,11} M$, 
$M$ represents every integer of the form $24n+11$. Furthermore since $M \prec_{24,11} L$, $L$ also represents
every positive integer of that form. Therefore  for any non negative integer $n$, there is an integer $a,b,c$ such that
$$
24n+11=a^2+10b^2+4bc+58c^2=a^2+2(b+5c)^2+8(b-c)^2.
$$ 
If $x=a, y=b+5c, z=b-c$ then one may easily show that $24n+11=x^2+2y^2+8z^2$ and $\gcd(xyz,6)=1$. 

\bigskip

\noindent {\bf (vi)  $p_5(x)+3p_5(y)+7p_5(z)$.} Let 
$$
M=\begin{pmatrix} 1&0&0\\0&3&0\\0&0&7\end{pmatrix}, \quad
N=\begin{pmatrix} 2&0&1\\0&3&0\\1&0&4\end{pmatrix}, \quad
L=\begin{pmatrix} 4&1&0\\1&7&0\\0&0&7\end{pmatrix}. 		
$$
Note that $\gen(M)=\{M,N\}$ and $S_{24,11} \subset Q(\gen(M))$. Since $N \prec_{8,3} M$,   
$M$ represents every positive integer of the form $24n+11$. Computation shows that
$$
R(M,24,11)-R_L(M,24,11)=\{ (4a,6b+3,\pm4a) \in (\z/24\z)^3: a \not \equiv 0 \pmod 3\}.
$$ 
By letting this set to $P_1$ and $\tau=\frac18\begin{pmatrix}1&0&21\\0&8&0\\-3&0&1\end{pmatrix}$, we may apply 
Corollary \ref{cor1}. Note that the eigen vector of $\tau$ is $(0,1,0)$, which is not in $R(N,24,11)$.  
Therefore every positive integer of the form $24n+11$ is represented by $L$. Let $a,b,c$ be integers such that
$$
24n+11=4a^2+7b^2+7c^2+2ab=(a-2b)^2+3(a+b)^2+7c^2.
$$ 
Now let $d=a-2b, e=a+b, f=c$. Note that $def \not \equiv 0 \pmod3$. By comparing both sides with modulo $8$, we have
$$
(d,e,f) \equiv (0,1,0), (1,1,1) \  \ \text{or} \  \ (0,0,1) \pmod 2.
$$
If all of $d,e$ and $f$ are odd, then there is an integer solution $x,y,z$ such that
$x^2+3y^2+7z^2=24n+11$ with $\gcd(xyz,6)=1$. Assume that $d$ and $e$ are all even. If $d=2g, e=2h$, then clearly $g \not \equiv h \pmod 2$. Furthermore
$$
d^2+3e^2+7f^2=(g\pm 3h)^2+3(g\mp h)^2+7f^2.
$$
Note that at least one of $g+h$ or $g-h$ is not divisible by $3$.  Therefore there is an integer solution $x,y,z$ such that
$x^2+3y^2+7z^2=24n+11$ with $\gcd(xyz,6)=1$. 

Finally assume that $d$ and $f$ are all even. In this case  $d \equiv f \pmod 4$ and $\gcd(e,6)=1$. 
Let $d=2i$ and $f=2j$ and assume that $i\equiv j \equiv 1 \pmod 2$. Then
$$
d^2+3e^2+7f^2=\left(\frac {3i\pm7j}2\right)^2+3e^2+7\left(\frac{i\mp 3j}2\right)^2.
$$
Note that at least one of $\frac {i-3j}2$ and $\frac{i+3j}2$ is odd and both of them is not divisible by $3$. For the remaining case, that is
$i \equiv j \equiv 0 \pmod 2$, one may use the following fact to reach the same consequence: since
$$
4(A^2+7B^2)=\left(\frac {3A\pm7B}2\right)^2+7\left(\frac{A\mp 3B}2\right)^2,
$$
there are odd integers $\alpha$ and $\beta$ such that $4(A^2+7B^2)=\alpha^2+7\beta^2$ under the assumption that $A,B$ are all odd. If $A \not \equiv B \pmod 2$,
then  there are odd integers $\alpha'$ and $\beta'$ such that $16(A^2+7B^2)=(\alpha')^2+7(\beta')^2$.      

\bigskip

\noindent {\bf (vii)  $p_5(x)+3p_5(y)+8p_5(z)$.} We first show that for any non negative integer $n$, there are integers $a,b,c$ such that
$$
2n+1=a^2+b^2+2c^2, \quad b\not \equiv c \pmod 2, \ \ a+b \not \equiv c \pmod 3 \ \ \text{and} \ \ b\not \equiv c \pmod 3.
$$
Note that $2n+1=a^2+b^2+2c^2$ has always a solution $(a,b,c) \in \z^3$.  
First assume that $n$ is even. Since $\langle 1,1,8\rangle$ represents $2n+1$ in this case, we may assume that $b$ is odd and $c$ is even. 
Since $b^2+2c^2 \ne 0$, we may assume that $b$ or $c$ is not divisible by $3$ by Lemma \ref{lem2}. Therefore we may find
integers $a,b,c$ satisfying all conditions given above by suitably choosing signs in the equation $2n+1=(\pm a)^2+(\pm b)^2+2(\pm c)^2$.
If $n$ is odd, we may assume that $b$ is even and $c$ is odd. 
Since the rest part of the proof is quite similar to the above, we are left to the readers. 

Now assume that $a,b,c$ are integers satisfying the above conditions. If $d=a, e=a+b+2c, f=c-a$, then
$$
\begin{array} {ll}
24n+12&=12(a^2+b^2+2c^2)=12(d^2+(-3d+e-2f)^2+2(d+f)^2)\\
&=(e-4d)^2+3e^2+8(e-4d-3f)^2.\end{array}
$$
Note that
$$
e+f\equiv 1 \pmod 2, \  \ e\not \equiv 0 \pmod 3 \ \ \text{and}  \ \  d-e \not \equiv 0 \pmod 3.
$$
Therefore if $x=e-4d, y=e, z=e-4d-3f$, then $24n+12=x^2+3y^2+8z^2$ and $\gcd(xyz,6)=1$.

\end{document}